\documentclass[a4paper,fleqn]{cas-sc}

\usepackage{natbib}
\usepackage{hyperref}
\usepackage{amsfonts}
\usepackage{mathrsfs}
\usepackage{amssymb}
\usepackage{amsmath}
\usepackage{amsthm}
\usepackage{enumerate}
\usepackage{color}
\usepackage{filecontents}
\usepackage{url}
\newtheorem{theorem}{Theorem}[section]
\newtheorem{remark}[theorem]{Remark}
\newtheorem{lemma}[theorem]{Lemma}

\numberwithin{equation}{section}
\newtheorem{example}[theorem]{Example}

\usepackage[justification=centering]{caption}
\DeclareMathAlphabet{\mathcal}{OMS}{cmsy}{m}{n}
\let\mathbb\relax
\DeclareMathAlphabet{\mathbb}{U}{msb}{m}{n}
\def\tsc#1{\csdef{#1}{\textsc{\lowercase{#1}}\xspace}}
\tsc{WGM}
\tsc{QE}

\begin{document}
	\let\WriteBookmarks\relax
	\def\floatpagepagefraction{1}
	\def\textpagefraction{.001}
	\let\printorcid\relax
	\shorttitle{Quasi-stationary distributions for continuous-time $\lambda$-recurrent    jump processes}    
	%
	%% Short author
	\shortauthors{Qian DU and Yong-Hua MAO}  
	
	\title [mode = title]{Quasi-stationary distributions for continuous-time $\lambda$-recurrent    jump processes}  
	
	%\tnotemark[<tnote number>] 
	%
	%% Title footnote 1.
	%% eg: \tnotetext[1]{Title footnote text}
	%\tnotetext[<tnote number>]{<tnote text>} 
	
	
	%\author[<aff no>]{<author name>}[<options>]
	%
	%% Corresponding author indication
	%\cormark[<corr mark no>]
	%
	%% Footnote of the first author
	%\fnmark[<footnote mark no>]
	%
	%% Email id of the first author
	%\ead{<email address>}
	%
	%% URL of the first author
	%\ead[url]{<URL>}
	
	%\credit{<Credit authorship details>}
	
	\author{Qian DU} %% Author name
	\author{Yong-Hua MAO}
	
	\affiliation{organization={School of Mathematical Sciences},
		addressline={Beijing Normal University,
			Laboratory of Mathematics and Complex Systems, Ministry of Education}, 
		city={	Beijing},
		%          citysep={}, % Uncomment if no comma needed between city and postcode
		postcode={100875}, 
		country={People's Republic of 	China}}
	
	%\author[<aff no>]{<author name>}[<options>]
	
	% Footnote of the second author
	%% Email id of the second author
	%\ead{}
	%
	%% URL of the second author
	%% Credit authorship
	%% Address/affiliation
	%\affiliation[<aff no>]{organization={},
		%%          citysep={}, % Uncomment if no comma needed between city and postcode
	%% Corresponding author text
	%\cortext[1]{Corresponding author}
	%
	%% Footnote text
	
	
	\begin{abstract}
		For the continuous-time  $\lambda$-recurrent jump process, the $\lambda$-recurrence assures the existence of quasi-stationary distribution when it has finite exit states (the states that have positive killing rates). And we give an explicit representation for this quasi-stationary distribution through $Q$-matrix, 
		where the components of the quasi-stationary distribution outside the  set $H$ of exit states can be represented by those within $H$. Sufficient condition is also provided for quasi-stationary distribution when the exit states are infinite.
	\end{abstract}
	
	
	%\item 
	%\item 
	%\item 
	
	% Keywords
	
	
	\maketitle
	
	% Main text
	%\section{}\label{}
	\section{Introduction}
	Quasi-stationary distribution (QSD, in short) was firstly discovered by 
	\citet{1947Certain} for Galton-Watson branching processes, now it is well known  the name of Yaglom limit. Actually for  classical branching processes, there exist a family of conditional limits including the Yaglom limit. For a finite continuous-time Markov chain, this conditional limit is an application of the famous Perron-Frobenius theorem for non-negative matrices, see \cite{mandl1960asymptotic,darroch1967quasi}. The notation ``quasi-stationary distribution'' comes from \citet{bartlett1960stochastic}. The same phenomena occur for some other kinds of processes, such as birth-death process \cite{cavender1978quasi,good1968limiting},
%	\cite{Daley1969,  coolen2006quasi, van2012conditions}, 
	skip-free process \citet{kijima1993quasi}, or diffusion process \cite{steinsaltz2007quasistationary,pinsky1985convergence,cattiaux2009quasi}. For a comprehensive bibliography, see
	\citet{pollettquasi,collet2013quasi}.
	
	Consider $Q$-matrix $Q=(q_{ij})$ on a  denumerable state space $\widetilde{E}=E\cup\{0\}$ with state $0$ absorbing. We assume that $Q$ is irreducible on $E$, that is, $\forall ~i,~j\in E$, $\exists$ different $i=i_0,~i_1,\cdots,~i_m=j\in E$ such that $q_{ii_1}q_{i_1i_2}\cdots q_{i_{m-1}j}>0$. We assume that $Q$ is totally stable: $\forall~i\in E,~q_i<\infty$, where $q_i=-q_{ii}$ by convention. And we also assume the corresponding $Q$-process $(\widetilde{X}_{t})_{t\geq0}$ is unique, with transition probabilities $P(t)=(p_{ij}(t))_{i,~j\in \widetilde{E}}$ such that $\widetilde p_{00}(t)=1$. 	Let $X=(X_t)_{t\geq0}$ be the jump process of  $(\widetilde{X}_t)_{t\geq0}$ restricted on $E$, with transition probabilities $P(t)=(p_{ij}(t))_{i,j\in E}$.
 We are going to study the QSD for $X$,  that is, we want to seek   $x>0$ and proper probability measure $u$ on $E$ such that,
	$$
	uQ=-x u ~~~~\text{or} ~~~~\sum\limits_{i\in E}u_ip_{ij}(t)=\exp(-xt) u_j,~~j\in E.
	$$
	Denote decay parameter 
	$$ 
	\lambda=\inf\{\epsilon>0:\int_{0}^{\infty}e^{\epsilon t}p_{ij}(t)dt=\infty\}=\sup\{\epsilon>0:\int_{0}^{\infty}e^{\epsilon t}p_{ij}(t)dt<\infty\},
	$$
	which  is independent of $i, ~j\in E$ by irreducibility. Denote $\tau_0$ the life time of $X$:~$\tau_0=\inf\{t\geq 0:\widetilde{X}_t=0\}$. Define $\lambda_0$ as the rate of convergence to zero of $\mathbb{P}_i[t<\tau_0]$, that is, 
	$$\lambda_0=\inf\{a>0:\mathbb{E}_ie^{a\tau_0}=\infty\}=\sup\{a:\mathbb{E}_ie^{a\tau_0}<\infty\},$$
	which  is also independent of $i\in E$ by irreducibility. It should be noticed that $0\leq\lambda_0 \leq \lambda $ since $p_{ii}(t)\leq\mathbb{P}_{i}[\tau_0>t]\leq1,$
	where each inequality can be strict (\cite{van2012quasi},~\cite[Remark 3.1.4]{jacka1995weak} or \cite[Page 190]{van2005extinction}).
	
	Recall that $X$ or $P$ is called $\lambda$-recurrent if $\int_{0}^{\infty}e^{\lambda t}p_{ij}(t)dt=\infty,~ \forall~i,~ j\in E;$ $\lambda$-positive recurrent if $\lim_{t\to\infty}e^{\lambda t}p_{ij}(t)>0$.
	It is well known (cf. \cite{kingman1963exponential}) that when $X$ is $\lambda$-recurrent,  it exists the unique $\lambda$-invariant measure  and $\lambda$-invariant vector  up to constant multiples, that is, there exist positive $x=(x_i)$ and $y=(y_i)$ such that 
	$$
	xP(t)=exp(-\lambda t)x,~~~~P(t)y=exp(-\lambda t)y.
	$$
		If there exists $\lambda$-QSD, then $\lambda$-recurrence is equivalent to $\mathbb{E}_i\lambda^{\tau_0}=\infty$, that is, $\lambda=\lambda_0$ (\cite[Corollary 9]{van2012quasi}). 
	However, we assume that $P$ is $\lambda$-recurrent, which allows us to  easily use the classical  theory of  minimal non-negative solutions (cf. \cite{hou2012homogeneous}).  
	``Embedded chain'' $T=(T_{ij})_{i,j\in E}$ is defined by 
		\begin{equation*}
		T_{ij}=\left\{
		\begin{aligned}
			\frac{q_{ij}}{q_i-\lambda},~i\not=j;\\
			0,~i=j.
		\end{aligned}
		\right.
	\end{equation*}
Denote convergence parameter
$$R=\sup\{s\geq 0:\sum_{n\geq 0}T_{ij}^{(n)}s^n<\infty\}=\inf\{s\geq 0:\sum_{n\geq 0}T_{ij}^{(n)}s^n=\infty\},$$ which is independent of $i,~j\in E$ by irreducibility.
For non-negative matrix $T$, \cite{seneta2006non} firstly gave the definition of its recurrence: $\sum_{n\geq 0}T_{ij}^{(n)}R^n=\infty$. In the following theorem, firstly we will build the equivalence between the $\lambda$-recurrence of $P(t)$ and recurrence of $T$, construct the $\lambda$-invariant measure and $\lambda$-invariant vector.
Similarly to taboo probability, for a finite set $H\subset E$, we define ${_{H}T_{ij}^{(n)}}$ recursively by
\begin{equation*}
	{_{H}T_{ij}^{(n)}}=\left\{
	\begin{aligned}
		&T_{ij},~~~~&n=1,\\
		&\sum_{l\notin H}{_{H}T_{il}^{(n-1)}}T_{lj},~~~&n\geq2,
	\end{aligned}
	\right.
\end{equation*}
with convention that ${_{H}T_{ij}^{(0)}}=\delta_{ij}$ if $i\notin H$, or 0 if $i\in H$.
	We call  state $i\in E$ an exit state for $X$ if~$q_{i0}>0.$ We give an explicit formula for $\lambda$-QSD by $Q$-matrix, where the components of the $\lambda$-QSD outside the  set $H$ of exit states can be represented in terms of the components within $H$. Our approach is helpful since most processes have more than one exit state, such as quasi-birth-and-death processes (cf. \cite{bean2000quasistationary,van2006alpha}), whose transition matrix is of a block-partitioned form. Denote
		$
		\tau_H^+=\inf\{t\geq \xi_1:X_t\in H,~0<s<t,~X_s\notin H\}
		$
		 the (first) return time to set $H$, where $\xi_1$ is the  time of the first jump: $\xi_1=\inf\{t> 0:X_t\not=X_0\}$.  Then we have the following theorem.
	\begin{theorem}\label{main}
		Assume $Q$ matrix is totally stable and its $Q$ process $P(t)$ is unique. Assume the absorption to state 0 is certain (i.e. $\mathbb{P}_i[\tau_0<\infty]=1,~\forall~i\in E$) and  $P(t)$ is $\lambda$-recurrent with $\lambda>0$. If the set of exit states 
		$H=\{i\in E:q_{i0}>0\}$ is finite, then there exists $\lambda$-QSD.
		More precisely, 	
		\begin{enumerate}
			\item [(1)] define $T^H=({T_{ij}^H})_{i,j\in H}$  by setting ${T_{ij}^{H}}=\sum_{n=1}^{\infty}{_HT_{ij}^{(n)}}.$ 
			Then there exists  the  unique  $0<\mu=(\mu_i)_{i\in H}~(\sum_{i\in H}\mu_i=1)$ that satisfies $\mu T^H=\mu$.
			\item [(2)] Let
			$$
			~~~~~~~~~~~~~~~~~~~x_i=\sum\limits_{j\in H}\frac{\mu_j}{q_i-\lambda}\sum\limits_{n=1}^{\infty}{_H{T_{ji}^{(n)}}},~~~~~~~~~i\notin H,
			$$
			and
			$$
			M_H:=\sum_{i\in H}\mu_i\mathbb{E}_ie^{\lambda\tau_0}I_{[\tau_H^+=\infty]}.
			$$
			Then by letting
			$u_i=\lambda\mu_i/((q_i-\lambda)M_H)$~for $i\in H$; $\lambda x_i/M_H$ for $i\notin H$, $u=(u_i)_{i\in E}$  is  $\lambda$-QSD of $P(t)$.
		\end{enumerate}
	\end{theorem}
	\begin{remark} 
	\begin{enumerate}
		\item [(i)] Fix $k\in E$, it is known that $\int_{0}^{\infty}e^{\lambda t}\mathbb{P}_{i}[X_t=j,\tau_j^+>t]dt,~{i\in E}$ is the $lambda$-invariant measure starting from a single state (see \cite{jacka1995weak}. Here we can also derive it by $Q$-matrix using the minimal non-negative  solution method, which make it easier to compute.
	And it follows from Lemma \ref{taboo relation} that 
		$$\int_{0}^{\infty}e^{\lambda t}\mathbb{P}_j[X_t=i,\tau_j^+>t]dt=\frac{1}{q_i-\lambda}\sum\limits_{n=1}^{\infty}{_{j}T_{ji}^{(n)}},~~~i\in E.$$
		\item [(ii)] 
		Similarly to (\ref{sum}), we can derive that the moment condition for the existence of $\lambda$-QSD is  $\mathbb{E}_ke^{\lambda\tau_{0}}I_{[\tau_k^+=\infty]}<\infty$.~
		\item [(iii)] The  existence of $\lambda$-QSD for $\lambda$-recurrent Markov chain with finite exit states is given in \cite[Theorem 12]{van2012quasi}. Here  we give the explicit formula for the $\lambda$-QSD.
	\end{enumerate}

\end{remark}

To deal with the case that exit states may be infinite, we use $h$-transform, which transforms $X$ to a new and ``equivalent'' Markov process $S(t)$  with single exit $\{k\}$, say. Let $h_i=\mathbb{P}_{i}[\tau_k^+<\infty]$ for $i\not=k$; $h_k=1$. We have following result.

\begin{theorem}\label{theorem 1.5}
	Assume the absorption to state 0 is certain (i.e. $\mathbb{P}_i[\tau_0<\infty]=1,~\forall~i\in E$) and  $P(t)$ is $\lambda$-recurrent, $\lambda>0$. Then
	$$\mathbb{E}_ke^{\lambda\tau_0}I_{[\tau_k^+=\infty]}\leq \frac{q_k(1-\mathbb{P}_k[\tau_k^+<\infty])}{\inf_{i\in E}h_i(q_k-\lambda)}.$$
	Particularly if $\inf_{i\in E}h_i>0$, then there exists the $\lambda$-QSD.
\end{theorem}
	In the following examples, we show that $\lambda$-recurrence or $\mathbb{E}_ie^{\lambda\tau_0}=\infty$ is not sufficient for the existence of $\lambda$-QSD.
	The first one is $\lambda$-positive recurrent and has $\lambda$-QSD.
	The second one is birth-death process on a line, and it is $\lambda$-null recurrent without $\lambda$-QSD. The third one reveals that $\lambda$-QSD exists even if the Markov chain is $\lambda$-transient. 
	We will prove them in Section 3.
\begin{example}
	Let $Q$ be defined on the state space $\mathbb{Z}_+=0\cup\{E\}$ with $q_0=0, ~q_{10}=w,~ q_{11}=r-1,~q_{i,~i+1}=p~(i\geq1),~q_{ii}=-1,~q_{i1}=q~(i>1)$, where $0<p<1,~ w>0, ~p+q=1,~ p+r+w=1$. Then
	\begin{itemize}
		\item [(1)] $\lambda=\dfrac{2(1-p)w}{1+w+\sqrt{(1-w)^2+4pw}}$;
		\item [(2)] $\left(\dfrac{p}{1-\lambda}\right)^{i-1}\dfrac{\lambda}{w},~i\geq 1$ is $\lambda$-QSD of $P(t)$;
		\item [(3)] $P(t)$ is $\lambda$-positive recurrent. 
	\end{itemize}
\end{example}
	\begin{example}\label{ex5.2}
		Let $Q$ be defined on the state space $E=\mathbb{Z}$ with $q_{i,i+1}=pc,~,q_{ii}=-c,~q_{i,i-1}=qc,$ where $0<p<1,~p+q=1$ and $c>0$. Then
		\begin{enumerate}
			\item [(1)] $P(t)$ is $\lambda$-null recurrent with $\lambda=(1-2\sqrt{pq})c$;
			\item [(2)] there is no corresponding $\lambda$-QSD.
		\end{enumerate}
	\end{example}
	\begin{example}\label{ex5.3}
		Let $Q$ on state space $\mathbb{Z}_{+}=0\cup\{E\}$ with $q_{00}=0,$~$q_{i,i+1}=pc,$~$q_{i,i-1}=qc~(i\geq1),$ where $0<p<q,~p+q=1$,~and $c>0$.
		Then
		\begin{enumerate}
			\item [(1)] $P(t)$ is $\lambda$-transient with $\lambda=(1-2\sqrt{pq})c$;
			\item [(2)] ${j(\sqrt{p})^{j-1}(\sqrt{q}-\sqrt{p})^2}/{\sqrt{q}^{j+1}},~j\geq1$ 	is  $\lambda$-QSD of $P$.
		\end{enumerate} 
	\end{example}
	\section{Quasi-stationary distribution}
	This section is devoted to presenting the  moment condition of the existence of $\lambda$-QSD for the $\lambda$-recurrent jump process, and then we apply it to some special cases.~Especially when the exit states are finite, $\lambda$-recurrence is sufficient for the existence of quasi-stationary distribution.
	
	Firstly, we will give the explicit representation for the $\lambda$-invariant measure and $\lambda$-invariant vector, based on the ``taboo probability'' for a finite set.
	For jump process  $\widetilde X_t$ on $\widetilde{E}=E\cup\{0\}$ with 
	transition probability $\widetilde P(t)=(\widetilde p_{ij}(t))_{i,~j\in \widetilde{E}}$, let
	$ P(t)=(p_{ij}(t))_{i,j\in E}$.
	
	For $s\geq0,~i,~j\in E$, let 
	$$
	P_{ij}(s)=\int_{0}^{\infty}e^{s t}p_{ij}(t)dt,~F_{ij}(\lambda)=\int_{0}^{\infty}e^{s t}d\mathbb{P}_i[\tau_j^+\leq t],
	$$
	where $\tau_j^+$ is the  first return time to $j$. Note that $P_{ij}(s),~F_{ij}(s)$ is finite for $s\in[0,\lambda)$. Then for $i\not=j$,
	\begin{equation}\label{generation function}
		P_{jj}(s)=\frac{1}{q_j-\lambda}\frac{1}{1-F_{jj}(s)},~P_{ij}(s)=\frac{F_{ij}(s)}{(q_j-\lambda)(1-F_{jj}(s))},~~~~s\in[0,\lambda).
	\end{equation}
Recall that ``embedded chain'' $T=(T_{ij})_{i,~j\in E}$ is defined by
	\begin{equation}\label{solution1}
	T_{ij}=\left\{
	\begin{aligned}
		\frac{q_{ij}}{q_i-\lambda},~i\not=j;\\
		0,~i=j.
	\end{aligned}
	\right.
\end{equation}
	It follows from irreducibility of $Q$ and \cite[Theorem 1]{kingman1963exponential} that $T$ is non-negative and irreducible on $E$.
To prove Theorem \ref{main}, we firstly build the equivalence between the $\lambda$-recurrence of $P(t)$ and recurrence of $T$.
\begin{lemma}\label{taboo relation}
For a finite set $H\subset E$,	fix $j\in H$, then
\begin{itemize}
	\item [(1)]\begin{equation}
		\begin{aligned}
			\int_{0}^{\infty}e^{\lambda t}d\mathbb{P}_{j}[\tau_H^+=\tau_i^+\leq t]&=\sum\limits_{n=1}^{\infty}{_{H}T_{ji}^{(n)}},~~~&i\in H,\\
			\int_{0}^{\infty}e^{\lambda t}\mathbb{P}_j[X_t=i,\tau_H^+>t]dt&=\frac{1}{q_i-\lambda}\sum\limits_{n=1}^{\infty}{_{H}T_{ji}^{(n)}},~~~&i\notin H;
		\end{aligned}
	\end{equation}
\item [(2)] $T$ is recurrent if and only if $P(t)$ is $\lambda$-recurrent;
\item [(3)] if $P(t)$ is $\lambda$-recurrent with $\lambda>0$, then 
$$u_i=\frac{\sum_{n=1}^{\infty}{{_{j}T_{ji}^{(n)}}}}{q_i-\lambda},~~~~i\in E$$
and $u=(u_i)_{i\in E}$ satisfies $uQ=-\lambda u$.
\end{itemize}

\end{lemma}
\begin{proof}
	(1) It follows from decomposing according to the last jump that 
	\begin{equation}
		\begin{aligned}
			\mathbb{P}_j[X_t=i,\tau_H^+>t]dt&=\delta_{ij}e^{-q_jt}+\sum_{k\notin H}\int_{0}^{t}\mathbb{P}_j[X_t=i,X_s=k,X_{s+ds}=i,X_v=i,v\in[s+ds,t],\tau_H^+>t]\\
			&=\delta_{ij}e^{-q_jt}+\sum_{k\notin H}\int_{0}^{t}\mathbb{P}_j[X_t=i,X_v=i,v\in[s+ds,t]|X_s=k,X_{s+ds}=i,\tau_H^+>s]\\
			&~~~~\cdot\mathbb{P}_j[X_{s+ds}=i|X_s=k]\cdot\mathbb{P}_j[X_s=k,\tau_H^+>s]~~~~~~~~~~~~~~\text{(by Markov property)}\\
			&=\delta_{ij}e^{-q_jt}+\sum_{k\notin H}\int_{0}^{t}e^{-q_i(t-s)}q_{ki}\mathbb{P}_j[X_t=k,\tau_H^+>s]ds.
		\end{aligned}
	\end{equation}
It follows from inverse Laplace transform that $$\int_{0}^{\infty}e^{\lambda t}\mathbb{P}_j[X_t=i,t<\tau_H^+]dt=\frac{\delta_{ij}}{q_i-\lambda}
		+\sum_{k\notin H}\frac{q_{ki}}{q_i-\lambda}\int_{0}^{\infty}e^{\lambda t}\mathbb{P}_j[X_t=k,t<\tau_H^+]dt.$$
	Set $$y_i^{(1)}=\frac{\delta_{ij}}{q_i-\lambda},~y_{i}^{(n+1)}=\sum\limits_{k\notin H}\frac{q_{ki}}{q_i-\lambda}y_{k}^{(n)},~n\geq 1,~i\notin H.$$
	Inductively,~$y_i^{(n)}=\frac{1}{q_i-\lambda}{_{H}T_{ji}^{(n-1)}}$.
	Using the second iteration method, 
	$$y_i=\sum_{n=1}^{\infty}y_i^{(n)}=\frac{1}{q_i-\lambda}\sum_{n=1}^{\infty}{_{H}T_{ji}^{(n)}}$$
	is the minimal solution of 
		\begin{equation}\label{solution3}
		y_i=\sum_{k\notin H}y_k\frac{q_{ki}}{q_i-\lambda}+\frac{\delta_{ij}}{q_i-\lambda},~~~~i\notin H.
	\end{equation}
		Similarly, it follows from decomposing according to the first jump that
		$$\int_{0}^{\infty}e^{\lambda t}d\mathbb{P}_{j}[\tau_H^+=\tau_i^+\leq t]=\frac{q_{ji}}{q_j-\lambda}
			+\sum_{k\notin H}\frac{q_{jk}}{q_j-\lambda}\int_{0}^{\infty}e^{\lambda t}\mathbb{P}_k[X_t=i,t<\tau_H^+]dt.$$
	This completes the proof following from the uniqueness of minimal non-negative solution. 
	
	(2) Particularly  we have
	$$\int_{0}^{\infty}e^{\lambda t}p_{ji}(t)=\frac{1}{q_i-\lambda}\sum_{n=0}^{\infty}T_{ji}^{(n)},~~~i\in E,$$
	which indicates the equivalence between the $\lambda$-recurrence of $P(t)$ and recurrence of $T$.
	
	(3) Denote $\bar{u}_i=\sum_{n=1}^{\infty}{{_{j}T_{ji}^{(n)}}},~i\in E$. It follows from (2) and \cite[Theorem 6.2]{seneta2006non} that 
	$$\sum_{i\not= k}\bar{u}_iT_{ik}=\bar{u}_k,~~~k\in E.$$
	That is,
	\begin{equation}
		\sum_{i\in E}\frac{\bar{u}_i}{q_i-\lambda}q_{ik}=-\lambda\frac{\bar{u}_k}{q_k-\lambda}.
	\end{equation}
	This completes the proof.
\end{proof}
\begin{lemma}\label{lemma 1}
	Assume $P(t)$  is $\lambda$-recurrent on $E$. For any finite subset $H\subset E$, define $T^H=({T_{ij}^H})_{i,~j\in H}$, where
	$$
	{{T}_{ij}^H}=\sum\limits_{n=1}^{\infty}{_HT_{ij}^{(n)}},~~~i,~j\in H.
	$$
	Then there exists the unique and positive $\mu=(\mu_i)_{i\in H} ~(\sum_{i\in H}\mu_i=1)$ that 
	satisfies 
	$$\mu T^H=\mu.$$ 
\end{lemma}
\begin{proof}
	It follows from the irreducibility of $Q$ that $T^H$ is non-negative and irreducible. Fix $k\in H$, decomposing according to the last entry into $H$,
	\begin{equation}
		{_kT_{kj}^{(n)}}
		=\sum\limits_{i\in H\backslash\{k\}}\sum\limits_{m=1}^{n-1}{_kT_{ki}^{(m)}}{_H{T_{ij}^{(n-m)}}}
		+{_HT_{kj}^{(n)}},~~~~n\geq1.
	\end{equation}
	Combined with recurrence of $T$ ($\sum_{n=1}^{\infty}{_kT_{kk}^{(n)}}=1$) we have
	\begin{equation}
		x_j^k
		=\sum\limits_{i\in H}x_i^k\sum\limits_{n=1}^{\infty}{_H{T_{ij}^{(n)}}},~~~j\in H,
	\end{equation}
	where $x_j^k=\sum_{n=1}^{\infty}{_kT_{kj}^{(n)}}$ for $j\in H$. Since $T^H$ is non-negative and irreducible, by letting 
	$\mu_i=x_i^k/\sum_{i\in H}x_i^k$ for $i\in H$,
	it follows from the Perron-Frobenius theorem that $\mu=(\mu_i)_{i\in H}$ is unique and positive. Furthermore, we can obtain that there exists the unique right eigenvector $v=(v_i)_{i\in H}$ (up to constant multiples) that satisfies $T^Hv=v$.
\end{proof}
	When $H$ is singleton ($H=\{k\}$ say), Lemma \ref{lemma 1} reduces to $F_{kk}(\lambda)=1.$ And it follows from Lemma \ref{taboo relation} (1) that
	\begin{equation}\label{sumT}
		\sum\limits_{i,~j\in H}\mu_iT_{ij}^H=\sum\limits_{i\in H}\mu_i\mathbb{E}_{i}
e^{\lambda\tau_H^+}I_{[\tau_H^+<\infty]}=1.
	\end{equation}
	 Now we are ready to prove Theorem \ref{main}.

	\begin{proof}
		(1) It follows from Lemma \ref{lemma 1} immediately.
			(2) Let
		\begin{equation}
			\bar x_i=\left\{
			\begin{aligned}
				&\mu_i,~~~&i\in H;\\
				&\sum\limits_{j\in H}\mu_j\sum\limits_{n=1}^{\infty}{_H{T_{ji}^{(n)}}},~~~&i\notin H.
			\end{aligned}
			\right.
		\end{equation}
		By using one-step decomposition,  we have for $j\in E$,
	\begin{equation}
		\begin{aligned}
			\sum\limits_{i\in E}\bar x_iT_{ij}
			&=\sum\limits_{i\in H}\mu_iT_{ij}+\sum\limits_{i\notin H}	\sum\limits_{r\in H}\mu_r\sum\limits_{n=1}^{\infty}{_HT_{ri}^{(n)}}T_{ij}\\
			&=\sum\limits_{i\in H}\mu_ip_{ij}+\sum\limits_{r\in H}\mu_r\left(\sum\limits_{n=2}^{\infty}{_HT_{rj}^{(n)}}\right)\\
			&=\sum\limits_{r\in H}\mu_r\sum\limits_{n=1}^{\infty}{_HT_{rj}^{(n)}}\\
			&=\bar x_j.
		\end{aligned}
	\end{equation}
It follows from Lemma \ref{taboo relation} (3) that $xQ=-\lambda u$, where $x_i=\bar{x}_i/(q_i-\lambda),~i\in E.$
	Similarly, we know that  
	\begin{equation}\label{y}
		y_i=\left\{
		\begin{aligned}
			&v_i,~~&i\in H;\\
			&\sum\limits_{j\in H}\sum\limits_{n=1}^{\infty}{_HT_{ij}^{(n)}}v_j,~~&i\notin H
		\end{aligned}
		\right.
	\end{equation}
	and $y=(y_i)_{i\in E}$ is the unique $\lambda$-invariant vector (up to constant multiples) of $P(t)$, where $v=(v_i)_{i\in H}$ satisfies $T^Hv=v$.
Furthermore, it follows from Lemma \ref{taboo relation} (2) that
\begin{equation}\label{sum}
	\begin{aligned}
		\sum_{i\in E}x_i&=\sum_{i\in H}\frac{\mu_i}{q_i-\lambda}+\sum_{i\notin H}\sum_{j\in H}\mu_j\int_{0}^{\infty}e^{\lambda t}\mathbb{P}_j[X_t=i,\tau_H^+>t]dt\\
		&==\sum_{i\in H}\frac{\mu_i}{q_i-\lambda}+\sum_{j\in H}\mu_j\left(\int_{0}^{\infty}e^{\lambda t}\mathbb{P}_j[X_t\in E,\tau_H^+>t]dt-\int_{0}^{\infty}e^{\lambda t}\mathbb{P}_j[X_t\in H,\tau_H^+>t]dt\right)\\
		&=\sum_{j\in H}\mu_j\mathbb{E}_j\int_{0}^{\tau_H^+\wedge \tau_0}e^{\lambda t}dt.
	\end{aligned}
\end{equation}
Combined with (\ref{sumT}) we have 
\begin{equation}\label{sumx}
	\sum\limits_{i\in E}x_i=\frac{1}{\lambda}\sum_{j\in H}\mu_j\mathbb{E}_je^{\lambda \tau_0}I_{[\tau_H^+=\infty]}=\sum_{j\in H}\frac{\mu_jq_{j0}}{\lambda(q_j-\lambda)}.
\end{equation}
By letting $u_i=\mu_i/M_H$ for $i\in H$; $x_i/M_H$ for $i\notin H$, $u=(u_i)_{i\in E}$ is $\lambda$-QSD of $P(t)$.
This completes the proof.
	\end{proof}

	When exit states may be infinite, we use the $h$-transform argument to give a sufficient condition under which the $\lambda$-recurrence implies the existence of $\lambda$-QSD.
	\begin{proof}
		(1)~Let $\tau_{k}=\inf\{t\geq \xi_1:X_t=k\}$~and $h_i=\mathbb{P}_{i}[\tau_{k}<\infty].$   Then
		\begin{equation}
				\left\{ 
			\begin{aligned}
				\bar{P}h(i)&=h_i,~~&i\not=k, \\
				h_k&=1,~~&i=k,
			\end{aligned}
			\right.
		\end{equation}
	where $\bar{P}=(\bar{p}_{ij})_{i,~j\in E}$ is the embedded chain of $P(t)$ with
	$$\bar{p}_{ij}=q_{ij}/q_i,~~~j\not=i;~~~\bar{p}_{ii}=0,~~~i,~j\in E.$$
	Using $h$-transform,
	define $s_{ij}=q_{ij}h_j/h_i,~i\in E$.
	Since
	\begin{align*}
		&\sum_{j\in E}s_{ij}=\sum_{j\in E}{q_{ij}h_j}/{h_i}=1,
		&\forall ~i\not=k,\\
		&\sum_{j\in E}s_{kj}=\sum_{j\in E}q_{kj}h_j=q_k[\mathbb{P}_k[\tau_k^+<\infty]-1]<0,
	\end{align*}
$S=(s_{ij})_{i,~j\in E}$ is a $Q$-matrix with only exit state  $\{k\}$.
	 Consider its minimal $Q$-process  $S(t)=\{s_{ij}(t):i,~j\in E\}$.
 It follows from the $\lambda$-recurrence of $P(t)$ that $S(t)$ is  also $\lambda$-recurrent noticing that  $s_{ij}(t)=p_{ij}(t)h_j/h_i$.  
	Let $\tilde{S}=\{\tilde{s}_{ij}:i,j\in \tilde{E}=E\cup 0\}$ be defined by
	\begin{equation*}
		\tilde s_{ij}=\left\{
		\begin{aligned}
			&0,~~~~&i=0=j;\\
			&1-\sum\limits_{j\not=0}s_{ij},~~&i\not=0,j=0;\\
			&s_{ij},~~&i\not=0,j\not=0.
		\end{aligned}
		\right.
	\end{equation*}
Denote $\tilde{\tau_0}$ the life time of $S(t)$. It follows from $s_{ik}=q_{ik}/h_i$ that
$$
\mathbb{P}_{i}[\tilde\tau_0<\infty]=\sum_{n=0}^{\infty}\bar{p}_{ik}^{(n)}(1-\mathbb{P}_k[\tau_k^+<\infty])/h_i,~~~i\in E.$$
By Lemma \ref{taboo relation}, apply (\ref{generation function})($s=0$) to derive that $\widetilde{S}(t)$ is also certainly absorbed at 0.
It follows from Theorem \ref{main} that
\begin{equation}\label{xh}
	x_i^h=\sum\limits_{n=1}^{\infty}{_{k}T_{ki}^{(n)}}h_i/(q_i-\lambda),~~~~i\in E
\end{equation}
and $x^h=(x_i^h)$ satisfies $x^hS=-\lambda x^h$. And from (\ref{sumx}) we have
\begin{equation}\label{xhsum}
	\sum_{i\in E}x_i^h=\frac{q_k(1-\mathbb{P}_k[\tau_k^+<\infty])}{\lambda(q_k-\lambda)}.
\end{equation}
Recall that $x_i=\sum\limits_{n=1}^{\infty}{_{k}T_{ki}^{(n)}}/(q_i-\lambda),~i\in E$~is the $\lambda$-invariant measure of $P(t)$ (see Remark 1.2 (i)). It follows from(\ref{xh}) and (\ref{xhsum}) that
$$\mathbb{E}_ke^{\lambda\tau_0}I_{[\tau_k^+=\infty]}=\lambda\sum_{i\in E}x_i\leq \frac{\sum_{i\in E}x_i^h}{\inf_{i\in E}h_i}=\frac{q_k(1-\mathbb{P}_k[\tau_k^+<\infty])}{ \inf_{i\in E}h_i(q_k-\lambda)}.$$
	Particularly if $\inf_{i\in E}h_i>0$, $x_i=\sum_{n=1}^{\infty}{_{k}T_{ki}^{(n)}}/(q_i-\lambda),~i\in E$ is $\lambda$-QSD of $P$.
	\end{proof}
	
	\section{Examples}
\begin{proof}[Proof of Example 1.4]
	(1) Obviously $Q$ is regular on $E$. Define $P=Q+I$, then
	$$
	F_{11}(s)=\sum\limits_{n=1 }^{\infty}{_{1}p_{11}^{(n)}}s^n=rs+\frac{pqs^2}{1-ps},~~~~~s\in(0,1/p).
	$$
	Solving equation $F_{11}(s)=1$ to derive $R_P$, where $R_P$ is the convergence parameter of $P$. It follows from \cite[Page 358]{kingman1963exponential} that
	$$\lambda=1-1/R_P=\dfrac{2(1-p)w}{1+w+\sqrt{(1-w)^2+4pw}}.$$
	(2) ``Embedded chain'' $T=(T_{ij})$ is
	\begin{equation*}T_{ij}=\left\{
		\begin{aligned}
			p/(1-\lambda-r),~~~&i=1,~j=2,\\
			p/(1-\lambda),~~~&i>1,~j=i+1,\\
			q/(1-\lambda),~~~&i>1,~j=1,\\
			0,~~~&\text{others}.
		\end{aligned}
	\right.
	\end{equation*}
Notice that $$T_{11}^{(1)}=0,~~~{_1T_{11}^{(n)}}=\dfrac{p}{1-\lambda-r}(\dfrac{p}{1-\lambda})^{n-2}\dfrac{q}{1-\lambda},~~~n\geq2.$$ It is easily to see that $\sum_{n\geq 0}{_1T_{11}^{(n)}}=1$, so $T$ is recurrent. It follows from Lemma \ref{taboo relation} (2) that $P(t)$ is $\lambda$-recurrent. Taking $j=1$,  it follows from Lemma \ref{taboo relation} (3) that
$$x_i=\frac{1}{q_i-\lambda}\sum_{n=0}^{\infty}{_{1}T_{1i}^{(n)}}=\dfrac{1}{1-\lambda-r}(\dfrac{p}{1-\lambda})^{i-1},~~i\in E$$
and $x=(x_i)_{i\in E}$ satisfies $xQ=-\lambda Q$.
Combined Theorem \ref{main} (2), we have 
$$u_i=(\dfrac{p}{1-\lambda})^{i-1}\dfrac{\lambda}{w},~~~i\in E$$
and $u=(u_i)_{i\in E}$ is $\lambda$-QSD of $P(t)$.
(3) It follows from (\ref{y}) that $y_1=1,~y_i=q/(q-\lambda),~i\in E$. $P(t)$ is $\lambda$-positive recurrent since $\sum_{i\geq1}x_iy_i<\infty$. 
\end{proof}
	Example \ref{ex5.2} and Example \ref{ex5.3} are birth-death processes on $\mathbb{Z}$ and $\mathbb{Z}_+$ respectively, which can be found in \cite[Page 345]{kingman1963exponential}, \cite[Page 188]{anderson2012continuous} or \cite{seneta1966quasi}.  Here  we prove that the birth-death process on a line is $\lambda$-null recurrent  without a corresponding $\lambda$-QSD and the birth-death on a half-line with an absorbing barrier at 0 is $\lambda$-transient but it has a corresponding $\lambda$-QSD.

	\begin{proof}[Proof of Example 1.5]
		(1) It follows from Example 5.3.1 in \cite{anderson2012continuous} or \cite{kingman1963exponential} that $\lambda=(1-2\sqrt{pq})c$.~(2) ``Embedded chain'' $T=(T_{ij})_{i,j\in \mathbb{Z}}$ is 
		\begin{equation}\label{T}
			T_{ij}=\left\{\begin{aligned}
				&\frac{\sqrt{p/q}}{2},~~~&j=i+1,\\
				&\frac{\sqrt{q/p}}{2},~~~&j=i-1,\\
				&0,~~&\text{others}.
			\end{aligned}
			\right.
		\end{equation}
	Note that $T_{jj}^{(m)}=0$ if $m$ is an odd integer, and 
	$$T_{jj}^{(2n)}=\binom{2n}{n}\left(\frac{\sqrt{p/q}}{2}\right)^n\left(\frac{\sqrt{q/p}}{2}\right)^n=\binom{2n}{n}\left(\frac{1}{4}\right)^n,~~~n\geq 0.$$
	It follows from Stirling's approximation that convergence parameter of $T$ is 1. And $\sum_{n\geq0}T_{jj}^{(n)}=\infty$, so $T$ is recurrent. It follows from Lemma \ref{taboo relation} (2) that $P(t)$ is $\lambda$-recurrent.
		Taking $k=0$, it follows that 
		$
		x_i=\sum_{n=1}^{\infty}{_0T_{0i}^{(n)}},~i \in \mathbb{Z}
		$
		satisfies $xT=x$.
		Noticing that $\sum_{n=1}^{\infty}{_{\{0,1\}}T_{11}^{(n)}}=(1/2)\sum_{n=1}^{\infty}{_1T_{11}^{(n)}}=1/2$, it follows from \cite[Page 202]{seneta2006non} that
		$$\sum_{n=1}^{\infty}{_0T_{11}^{(n)}}=\frac{1}{1-\sum_{n=1}^{\infty}{_{\{0,1\}}T_{11}^{(n)}}}=2.$$
		So 
		$$x_1=\sum_{n=1}^{\infty}{_0T_{01}^{(n)}}=\frac{1}{2}\sqrt{\frac{p}{q}}\sum_{n=1}^{\infty}{_0T_{11}^{(n)}}=\sqrt{\frac{p}{q}}.$$
	Inductively,
$	x_j=(\sqrt{p/q})^j,~j\in \mathbb{Z}$ is the unique invariant measure of $T$ up to constant multiples. By letting $$u_i=x_i/(q_i-\lambda)=
\frac{1}{2c\sqrt{pq}}(\sqrt{\frac{p}{q}})^i,$$
it follows from Lemma \ref{taboo relation} (3) that $uQ=-\lambda u.$
	Furthermore, there is no $\lambda$-QSD of $P(t)$ since $(u_i)_{i\in \mathbb{Z}}$ is non-summable.
	It follows from symmetric property of $T$ that 
	$
	y_j=(\sqrt{q/p})^j,~j\in \mathbb{Z}
	$,~$y=(y_i)_{i\in E}$ satisfies $Ty=y$ or $Qy=-\lambda y$,
	which indicates that  $\sum_{i\in \mathbb{Z}}u_iy_i=\infty.$ That is $P(t)$ is $\lambda$-null recurrent.
\end{proof}
\begin{proof}[Proof of Example 1.6]
	(1) It follows from \cite[Proposition 3.2]{anderson2012continuous} or \cite{tweedie1974some} that $\lambda=(1-2\sqrt{pq})c$.	(2) Let $\bar{T}=(\bar{T}_{ij})_{i,~j\in E}$ be the ``Embedded chain''  $T$ in (\ref{T}) restricted on $E$. Since $\bar{T}$ has only half the number of paths returning to 1 compared with $T$, it follows that 
	$\sum_{n=1}^{\infty}{_1\bar T_{11}^{(n)}}=1/2.$ So $P(t)$ is $\lambda$-transient.
	Without loss of generality, we set $x_1=1$. It follows from 
	$$\sum\limits_{i\geq1}x_iq_{ij}=-(1-2\sqrt{pq})cx_j,~~~j\geq1$$
	inductively that  
	$
	x_j=j(\sqrt{p/q})^{j-1},~j\geq1.
	$
	It is easy to get
	$
	\sum_{j\geq1}x_j={q}/{(\sqrt{q}-\sqrt{p})^2}.
	$
	By letting 
	$$
	x_i={j(\sqrt{p})^{j-1}(\sqrt{q}-\sqrt{p})^2}/{\sqrt{q}^{j+1}},~~~j\in E,
	$$
	$x=(x_i)_{i\in E}$ is  $\lambda$-QSD of $P(t)$  even if $P(t)$ is $\lambda$-transient.
\end{proof}

\bibliographystyle{elsarticle-harv}
%%
%%% Loading bibliography database
\bibliography{rerefer.bib}
%
%% Biography
%% Here goes the biography details.
%% Here goes the biography details.

\end{document}